\documentclass[11pt,reqno]{amsart}
\usepackage{enumerate}
\usepackage{color}
\usepackage{graphicx}
\usepackage{hyperref}
\usepackage{amsmath}

\newtheorem{theorem}{Theorem}
\newtheorem{proposition}{Proposition}

\newtheorem{remark}[proposition]{Remark}
\newtheorem{lemma}[proposition]{Lemma}
\newfont{\bb}{msbm10 at 12pt}

\numberwithin{equation}{section}

\newcommand{\bal}{\begin{align}}      \newcommand{\eal}{\end{align}}
\newcommand{\ba}{\begin{array}}      \newcommand{\ea}{\end{array}}
\newcommand{\bc}{\begin{center}}     \newcommand{\ec}{\end{center}}
\newcommand{\be}{\begin{enumerate}}  \newcommand{\ee}{\end{enumerate}}
\newcommand{\beQ}{\begin{eqnarray*}} \newcommand{\eeQ}{\end{eqnarray*}}
\newcommand{\bi}{\begin{itemize}}    \newcommand{\ei}{\end{itemize}}
\newcommand{\bt}{\begin{tabular}}    \newcommand{\et}{\end{tabular}}
\newcommand{\bdm}{\begin{displaymath}} \newcommand{\edm}{\end{displaymath}}

\def\<{\langle}     
\def\>{\rangle}

\newcommand{\D}{D\!\!\!\!/\,}

\newcommand{\mult}{\gamma\!\!\!/}

\newcommand{\RSB}{\mathbf{S}\!\!\!\!/\,}
\newcommand{\Ss}{R\!\!\!\!/\,}
\newcommand{\nbs}{\nabla\!\!\!\!/\,}

\newcommand{\gb}{g\!\!\!/}
\newcommand{\APSm}{\mathcal{P}^-_{>0}}

\begin{document}

\title[Rigidity of the Riemannian Schwarzschild manifold]{A spinorial proof of the rigidity of the Riemannian Schwarzschild manifold} 
      
\author{Simon Raulot}
\address[Simon Raulot]{Laboratoire de Math\'ematiques R. Salem
UMR $6085$ CNRS-Universit\'e de Rouen
Avenue de l'Universit\'e, BP.$12$
Technop\^ole du Madrillet
$76801$ Saint-\'Etienne-du-Rouvray, France.}
\email{simon.raulot@univ-rouen.fr}

\begin{abstract}
We revisit and generalize a recent result of Cederbaum \cite{C2,C3} concerning the rigidity of the Schwarz\-schild manifold for spin manifolds. This includes the classical black hole uniqueness theorems \cite{BM,GIS,Hw} as well as the more recent uniqueness theorems for photon spheres \cite{C,CG1,CG2}. 
\end{abstract}


\date{\today}   

\maketitle 
\pagenumbering{arabic}


\section{Introduction}


An $(n+1)$-dimensional {\em vacuum} spacetime is a Lorentzian manifold $({\mathcal L}^{n+1},
\mathfrak{g})$ satisfying the Einstein field equations $\mathfrak{Ric}=0$, where $\mathfrak{Ric}$ is the Ricci tensor of the metric $\mathfrak{g}$. The vacuum is said to be {\em static} when 
\begin{eqnarray*}
{\mathcal L}^{n+1}=\mathbb{R}\times M^n,\qquad \mathfrak{g}=-N^2\,dt^2+g,
\end{eqnarray*}
where $(M^n,g)$ is an $n$-dimensional connected smooth Riemannian manifold, that we will take to be orientable, standing for the unchanging slices of constant time and $N\in C^\infty (M^n)$ is a  non-trivial smooth function on $M^n$. To model the exterior of an isolated system, it seems physically natural to require asymptotic flatness, that is, the Cauchy hypersurface $M^n$ is usually taken to be asymptotically flat. The vacuum Einstein field equations can be translated into the following two conditions on $(M^n,g)$ and the lapse function $N$: 
\begin{equation}\label{vacuum}
\nabla^2N=N\,Ric,\qquad \Delta N=0,
\end{equation}
where $Ric$, $\nabla$ and $\Delta$ are respectively the Ricci tensor, the covariant derivative and the Laplace operator of the Riemannian manifold $(M^n,g)$. Taking traces in the first of these two equations and taking into account the second one, we conclude immediately that the scalar curvature of $(M^n,g)$ is zero. It is usual to call the triple $(M^n,g,N)$ a {\em static vacuum triple}. 

The $(n+1)$-dimensional Schwarzschild spacetimes \cite{Schw,Tang} are a $1$-parameter family of static, spherically symmetric and asymptotically flat solutions to the vacuum Einstein field equations. For a parameter $m\in\mathbb{R}$, it is given by the static vacuum triple $\big(\mathbb{R}^n\setminus\overline{B_{r_m}(0)},g_m,N_m\big)$ where the metric $g_m$ and the lapse function are 
\begin{eqnarray*}
g_m=N^{-2}_mdr^2+r^2g_\mathbb{S},\quad N_m(r)=\Big(1-\frac{2m}{r^{n-2}}\Big)^{1/2}
\end{eqnarray*} 
with $g_\mathbb{S}$ denoting the standard metric on $\mathbb{S}^{n-1}$ and $r_m:=(2m)^{1/(n-2)}$ for $m>0$ and $r_m:=0$ for $m\leq 0$. 

For $m>0$, this spacetime represents the exterior of a black hole whose event horizon occurs at $r=r_m$. A striking result due to Bunting and Masood-ul-Alam \cite{BM} for $n=3$ (generalizing the seminal work by Israel \cite{I}) states that the Schwarzschild spacetime is in fact the only asymptotically flat static vacuum spacetime with nondegenerate horizons. For $n\geq 4$, this result has been generalized in \cite{Hw,GIS}.  In particular, there exist no asymptotically flat static vacuum spacetimes with multiple black holes. The proof of this result proceeds as follows: first it is shown that the metric $g$ can be conformally deformed using the lapse function $N$ to get two Riemannian metrics, one which can be compactified by adding a point at infinity and another one which is asymptotically flat with zero ADM mass. Then gluing these two manifolds along their boundaries gives an asymptotically flat Riemannian manifold with zero scalar curvature and with zero mass. The rigidity part of the positive mass theorem for non-smooth metrics \cite{b,MS} then applies and allows to conclude that $(M^n,g)$ is conformally flat. For $n=3$, this is enough to conclude while for $n\geq 4$ an additional argument has to be used but in any case, the initial static vacuum triple has to be isometric to $\big(\mathbb{R}^n\setminus\overline{B_{r_m}(0)},g_m,N_m\big)$.

More recently, new uniqueness theorems for photons spheres have been studied using this approach. A timelike hypersurface $P^3$ in a static vacuum spacetime $({\mathcal L}^{4},\mathfrak{g})$ is called a photon sphere if it is a totally umbilical hypersurface and if the associated lapse function is constant on each of its connected components. In the Schwarzschild spacetime, there is only one photon sphere given by $\{r=3m\}$ which models photons spiraling around the central black hole ``at a fixed distance''. In fact, as shown in \cite{CG1}, the Schwarzschild spacetime is the unique asymptotically flat static vacuum spacetime with photon sphere as an inner boundary. The main problem to apply the method of Bunting and Masood-ul-Alam in this situation comes mainly from the fact that the gluing hypersurfaces are not totally geodesic so that the gluing process does not work directly. To overcome this difficulty, Cederbaum and Galloway begin by gluing in a $C^{1,1}$ fashion some pieces of Schwarzschild time-slices of well-chosen masses on each photon spheres. The resulting manifold has only totally geodesic inner boundary components and then the method in \cite{BM} can be used. A generalization of this result for higher dimensional static vacuum triple has recently been addressed in \cite{CG2}. 

In \cite{C2,C3}, Cederbaum proves that both static vacuum black hole and photon sphere uniqueness theorems can be deduced from a more general rigidity result for the Riemannian time-slice of the Schwarzschild spacetime. This statement deals with pseudo-static systems $(M^n,g,N)$  which generalize the notion of static vacuum triple since they do not need to satisfy the full set of the static equations (\ref{vacuum}). On the other hand, since the black hole as well as the photon sphere boundary conditions arise from Lorentzian geometric considerations, they have to be translated into purely Riemannian assumptions. This is done in \cite{C2,J} and this give rises to the notions of {\it nondegenerate static horizons} and {\it quasilocal photon surfaces}. With these definitions, Cederbaum is able to prove the following general rigidity result:
\begin{theorem}
Let $(M^n,g,N)$ be an asymptotically isotropic pseudo-static system of mass $m$ with $n\geq 3$. Assume $M^n$ has a compact inner boundary whose components are either nondegenerate static horizons or quasilocal photon surfaces. Then $m>0$, $(M^n,g)$ is isometric to a suitable piece of the Schwarzschild manifold of mass $m$ and $N$ coincides with $N_m$.  
\end{theorem}

In this paper, we address another approach to this problem using spinors. Although we have to assume that the manifold is spin (which is automatically satisfied in the $3$-dimensional case), we recover the full result of Cederbaum and even allow to relax the quasilocal photon surface condition. This new type of boundary condition will be referred to as a {\it generalized quasilocal photon surface}. Moreover, our arguments also avoid all the gluing constructions which are the delicate part of her proof and our results also include the black hole \cite{BM,GIS,Hw} and the photon spheres uniqueness theorems \cite{C,CG1,CG2}. This is done by using a positive mass theorem for manifolds with inner boundary due to Herzlich \cite{he1,he2} which can be applied if the first eigenvalue of the boundary Dirac operator satisfies a certain lower bound. As we shall see one can check that this lower bound is fulfilled using both the Friedrich inequality \cite{F} and a generalization of an inequality of  Hijazi-Montiel-Zhang \cite{HMZ}. We then get:
\begin{theorem}\label{MainTheorem}
Let $(M^n,g,N)$ be a spin asymptotically isotropic pseudo-static system of mass $m$ with $n\geq 3$. Assume $M^n$ has a compact inner boundary whose components are either nondegenerate static horizons or generalized quasilocal photon surfaces. Then $m>0$, $(M^n,g)$ is isometric to a suitable piece of the Schwarzschild manifold of mass $m$ and $N$ coincides with $N_m$. 
\end{theorem}

The idea to use the positive mass theorem of Herzlich for the black hole uniqueness problem was suggested by Walter Simon in \cite{Simon} and I am very grateful to him for allowing us to reproduce his unpublished alternative proof of the generalization of Israel's theorem \cite{I} by \cite{MZH,R} (see Appendix \ref{Simon}). It is also a pleasure to thank him for his careful reading as well as for his valuable comments of a previous version of this paper. We end this work by noticing that the use of the positive mass theorem can be dropped from the Simon's approach in the context of $3$-dimensional static vacuum triples (see Appendix \ref{UniquePMT-horizon}). However, since it is essential to deal with the general assumptions of Theorem \ref{MainTheorem}, it is important for the author to include this proof here. In Appendix \ref{UniquePMT-photon}, we see that this method is suitable to get a new proof of the uniqueness of a connected photon surface in this setting. Carla Cederbaum informed me that a work \cite{CF} with her student Axel Fehrenbach, in which they get similar results using arguments \`a la Robinson \cite{R}, is in progress and I would like to thank her for this.  

Finally, it is a pleasure to thank Piotr Chru\'sciel for his invitation to the seminar of the Gravitational Physics team of the University of Vienna as well as for his hospitality and where this work began.   


\section{The setting}\label{Setting}


Here we consider a much broader class than the static vacuum triple, namely the {\it pseudo-static} system. Such a system is defined by a triple $(M^n,g,N)$ where $M^n$ is an $n$-dimensional smooth manifold endowed with a smooth Riemannian metric $g$ with nonnegative scalar curvature $R$ and $N$ is a nonnegative smooth harmonic function on $M^n$. It is then immediate from (\ref{vacuum}) that a static vacuum triple is a pseudo-static system. In the rest of this article, we will always assume that $N>0$ away from the boundary $\partial M^n$ of $M^n$ if it exists.

On the other hand, we will use the following definition of {\it asymptotically isotropic} manifolds (see \cite{CG1,J}). An $n$-dimensional smooth Riemannian manifold $(M^n,g)$, $n\geq 3$, is {\it asymptotically isotropic of mass $m$} if the manifold $M^n$ is diffeomorphic to the union of a (possibly empty) compact set and an open end $E^n$ which is diffeomorphic to $\mathbb{R}^n\setminus\overline{B}$, where $B$ is an open ball in $\mathbb{R}^n$, and if there exists a constant $m\in\mathbb{R}$ such that, with respect to the coordinates $(y^i)$ induced by this diffeomorphism, we have
\begin{eqnarray*}
 g_{ij} =  (g_m)_{ij} +O_2\big(s^{1-n}\big)
\end{eqnarray*}
for $i,j=1,...,n$ on $\mathbb{R}^n\setminus\overline{B}$ as $s:=\sqrt{(y^1)^2+...+(y^n)^2}\rightarrow\infty$. Here
\begin{eqnarray*}
g_m  :=  \Big(1+\frac{m}{2s^{n-2}}\Big)^{\frac{4}{n-2}}\delta
\end{eqnarray*}
denotes the spatial Schwarzschild metric in isotropic coordinates with $\delta$ the flat metric on $\mathbb{R}^n$. Moreover for such a manifold, a smooth function $N:M^n\rightarrow\mathbb{R}$ is called an {\it asymptotic isotropic lapse of mass $m$} if it satisfies 
\begin{eqnarray*}
N=\widetilde{N}_m+O_2(s^{1-n})
\end{eqnarray*}
on $\mathbb{R}^n\setminus\overline{B}$ as $s\rightarrow\infty$ with respect to the same diffeomorphism, coordinates and mass $m$ described above. Here, $\widetilde{N}_m$ denotes the Schwarzschild lapse function in isotropic coordinates, given by
\begin{eqnarray*}
\widetilde{N}_m(s)=\frac{1-\frac{m}{2s^{n-2}}}{1+\frac{m}{2s^{n-2}}}.
\end{eqnarray*}
A triple $(M^n,g,N)$ is called an {\it asymptotically isotropic system of mass $m$} if $(M^n,g)$ is an asymptotically isotropic manifold of mass $m$ and $N$ is an asymptotic isotropic lapse of same mass $m$. 

\begin{remark}
A $3$-dimensional asymptotically flat static vacuum triple (in the sense of (\ref{AsympFlat})) is automatically an asymptotically isotropic system from the work of Kennefick and \'O Murchadha \cite{KOM}.
\end{remark}

The statement of Theorem \ref{MainTheorem} assumes boundary conditions which we now make more precise. As mentioned in the introduction, although they are here expressed only in terms of Riemannian geometry, these boundary conditions arise from Lorentzian notions of horizons and photon spheres. We refer to \cite{C2,J} where these characterizations are derived. Assume that $M^n$ has a compact inner boundary $\partial M^n=\coprod_{i=1}^k\Sigma_i$ where $\Sigma_i$ denotes one of its connected components for $1\leq i\leq k$ and let $\nu$ be the unit normal to $\partial M^n$ pointing toward infinity. A boundary component $\Sigma_i$ of $\partial M^n$ is said to be a static horizon if it is a totally geodesic component of the zero level set of the lapse $N$. It is called nondegenerate if its normal derivative is a  nonzero constant. In the following, such a hypersurface will be referred to as a {\it nondegenerate static horizon} and in particular, it satisfies:
\begin{eqnarray}\label{SemiStatic}
H_i=0,\quad N_i:=N_{|\Sigma_i}=0,\quad \nu_i(N):=\frac{\partial N}{\partial \nu}_{\big|\Sigma_i}> 0,
\end{eqnarray}
where $H_i$ is the mean curvature of $\Sigma_i$ in $(M^n,g)$. In our conventions, the mean curvature of an $(n-1)$-dimensional round sphere seen as the inner boundary of the exterior of an $n$-dimensional Euclidean ball is $n-1$. On the other hand, we will say that $\Sigma_i$ is a {\it quasilocal photon surface} if it is totally umbilical, if $N_i$ as well as $H_i$ are positive constants and if there exist a constant $c_i>1$ such that
\begin{eqnarray}\label{QuasiLocalPhotonSpheres1}
\Ss_{i}&=&\frac{n-2}{n-1}c_i H_i^2
\end{eqnarray}
and
\begin{eqnarray}\label{QuasiLocalPhotonSpheres2}
2\nu_i(N) &=&\frac{n-2}{n-1}\Big(c_i-1\Big)H_i N_i.
\end{eqnarray}
Here $\Ss_i$ denotes the scalar curvature of $\Sigma_i$ with respect to the induced metric $\gb_i:=g_{|\Sigma_i}$ which has to be constant because of (\ref{QuasiLocalPhotonSpheres1}). Note that the constant $c_i$ differs from \cite{C2,C3} by a multiplicative constant. Obviously, the intersection of the photon sphere $\{r=3m\}$ with a time-slice $\{t=const.\}$ in the Schwarzschild spacetime fits into this class of hypersurfaces. In the following, instead of the assumption (\ref{QuasiLocalPhotonSpheres1}), we will only assume that 
\begin{equation}\label{GeneralizedQPS}
\begin{array}{rll}
\Ss_{i} & \geq & \frac{n-2}{n-1}c_i H_i^2.
\end{array}
\end{equation}
In particular, the scalar curvature is not assumed to be constant on $\Sigma_i$. Therefore a quasilocal photon surface for which (\ref{QuasiLocalPhotonSpheres1}) is relaxed to (\ref{GeneralizedQPS}) will be referred to as a {\it generalized quasilocal photon surface}.


\section{The spinorial tools}\label{SpinTools}


In this section, we recall results from spin geometry which are needed to prove Theorem \ref{MainTheorem}. For more details on this wide subject we refer to the classical monographs \cite{BHMM,LM} and the references therein. 

On a $n$-dimensional Riemannian spin manifold $(M^n,g)$ with boundary, there exists a smooth Hermitian vector bundle over $M^n$ called the spinor bundle which will be denoted by $\mathbf{S}$. The sections of this bundle are called spinors. Moreover, the tangent bundle $TM$ acts on $\mathbf{S}$ by Clifford multiplication $X\otimes \psi\mapsto \gamma(X)\psi$ for any tangent vector fields $X$ and any spinor fields $\psi$. On the other hand, the Riemannian Levi-Civita connection $\nabla$ lifts to the so-called spin Levi-Civita connection (also denoted by $\nabla$) and defines a metric connection on $\mathbf{S}$ that preserves the Clifford multiplication. The Dirac operator is then the first order elliptic differential operator acting on the spinor bundle $\mathbf{S}$ given by $D:=\gamma\circ\nabla$. The spin structure on $M^n$ also induces (via the unit normal field to $\partial M^n$) a spin structure on its boundary. This allows to define the {\it extrinsic} spinor bundle $\RSB:=\mathbf{S}_{|\partial M^n}$ over $\partial M^n$ on which there exists a Clifford multiplication $\mult$ and a metric connection $\nbs$. Similarly, the extrinsic Dirac operator is defined by taking the Clifford trace of the covariant derivative $\nbs$ that is $\D:=\mult\circ\nbs$. From the spin structure on $\partial M^n$, one can also construct an {\it intrinsic} spinor bundle for the induced metric $\gb$, denoted by $\mathbf{S}^\partial$, and endowed with a Clifford multiplication $\gamma^{\partial}$ and a spin Levi-Civita connection $\nabla^{\partial}$. Note that the (intrinsic) Dirac operator on $(\partial M^n,\gb)$ is obviously defined by $D^{\partial}=\gamma^{\partial}\circ\nabla^{\partial}$. In fact, we have an isomorphism 
$$
\big(\RSB,\nbs,\mult\big)\simeq
\left\lbrace
\begin{array}{ll}
\big(\mathbf{S}^{\partial},\nabla^{\partial},\gamma^{\partial}\big) & \text{ if } n \text{ is odd}\\
\big(\mathbf{S}^{\partial},\nabla^{\partial},\gamma^{\partial}\big)\oplus\big(\mathbf{S}^{\partial},\nabla^{\partial},-\gamma^{\partial}\big)& \text{ if } n \text{ is even}
\end{array}
\right.
$$
so that the restriction of a spinor field on $M^n$ to $\partial M^n$ and the extension of a spinor field on $\partial M^n$ to $M^n$ are well-defined. These identifications also imply in particular that the spectrum of the extrinsic Dirac operator is an intrinsic invariant of the boundary: it only depends on the spin and Riemannian structures of $\partial M^n$ and not on how it is embedded in $M^n$. The first nonnegative eigenvalue of the extrinsic Dirac operator, which corresponds to the lowest eigenvalue (in absolute value) of $D^{\partial}$, will be denoted by $\lambda_1(\D)$. 
 

\subsection{Herzlich's positive mass theorem for spin manifolds with boundary}


One of the main result needed in our approach is a sharp version of the positive mass theorem for asymptotically flat spin manifold with boundary due to Herzlich \cite{he1,he2}. It is important to note that the spin assumption is not only assumed just to adapt the Witten approach \cite{W} to this setting.  As we shall briefly recall below, the choice of the boundary condition under which the Dirac operator is studied is crucial to get rigidity. Note that a positive mass theorem for asymptotically flat manifolds with compact inner boundary has recently been obtained without the spin assumption by Hirsch and Miao \cite{HM}. 

Recall that a Riemannian manifold $(M^n,g)$ is said to be asymptotically flat if the complement of some compact set is diffeomorphic to the complement of a ball in $\mathbb{R}^n$ and the difference between the metric $g$ and the Euclidean metric $\delta$ in this chart behaves like $s^{-\tau}$, its first derivatives like $s^{-\tau-1}$ and its second derivative like $s^{-\tau-2}$ where $\tau>(n-2)/2$. If moreover the scalar curvature is integrable, its ADM mass, defined by
\begin{eqnarray*}
m_{ADM}(g):=\frac{1}{2(n-1)\omega_{n-1}} \lim_{r\rightarrow\infty} \int_{S_r}\Big({\rm div}_\delta g-d\big({\rm tr}_\delta g\big)\Big)(\nu_r), 
\end{eqnarray*}
is a geometric invariant of $(M^n,g)$ by Bartnik \cite{b} and Chru\'sciel \cite{Chrusciel} independently. Here $S_r$ is a coordinate sphere of radius $r$ with $\nu_r$ as its unit normal vector field pointing towards infinity, $\delta$ is the Euclidean metric and $\omega_{n-1}$ is the volume of the standard $(n-1)$-sphere in $\mathbb{R}^n$. Moreover, the volume integral is with respect to the Euclidean metric and ${\rm div}_\delta$ (resp. ${\rm tr}_\delta$) is the divergence (resp. the trace) with respect to this metric. Obviously, an asymptotically isotropic manifold of mass $m$ as defined in Section \ref{Setting} is an asymptotically flat manifold with ADM mass equals to $m$. The positive mass theorem asserts that if in addition to all the previous assumptions, the scalar curvature is nonnegative then the ADM mass is also nonnegative and if it is zero, $(M^n,g)$ must be isometric to the Euclidean space. This result was first proved by Schoen and Yau \cite{SY1,SY2} for $3$-dimensional manifolds and thereafter, they showed how there method can be used for dimensions less than eight. In a recent preprint \cite{SY3}, the higher-dimensional cases have been treated. On the other hand, Witten \cite{W} discovered a proof with a completely different method relying on spin geometry which we now discuss. 

In the spin setting and if the manifold has a compact inner boundary, the proof of the positive mass theorem relies on the existence of $\psi\in\Gamma(\mathbf{S})$ (in some weighted Sobolev or H\"older spaces) such that 
\begin{eqnarray*}
D\psi=0\quad\text{and}\quad\psi\rightarrow\psi_0
\end{eqnarray*}
where $\psi_0\in\Gamma(\mathbf{S})$ is constant near infinity and where the boundary condition on $\partial M^n$ has to be well-chosen. Then, integrating by parts the famous Schr\"o\-din\-ger-Lichnerowicz formula on large domains $\Omega_r:=\{s\leq r\}$ with $r>0$ and taking the limit as $r\rightarrow\infty$ leads to 
\begin{eqnarray*}
\frac{1}{2}(n-1)\omega_{n-1}m_{ADM}(g)=\int_M\Big(|\nabla\psi|^2+\frac{R}{4}|\psi|^2\Big)-\sum_{i=1}^k\int_{\Sigma_i}\<\D_i\psi_i+\frac{H_i}{2}\psi_i,\psi_i\>
\end{eqnarray*}
where $\D_i$ is the restriction of $\D$ to $\RSB_i:=\RSB_{|\Sigma_i}$ and $\psi_i=\psi_{|\Sigma_i}$ for all $i\in\{1,...,k\}$. Since the scalar curvature is assumed to be nonnegative, it turns out that the right-hand side of this expression is nonnegative if one can ensure that each boundary term is nonpositive. This can be done by imposing the Atiyah-Patodi-Singer boundary condition on $\psi_{|\partial M^n}$ as well as an additional assumption on the first eigenvalue of the boundary Dirac operator. We refer to the original papers of Herzlich (and to \cite{bca} for a more detailed treatment of the analytic part) for a rigorous proof of this result. A straightforward adaptation of these arguments allows to obtain the following version of Herzlich's positive mass theorem (compare with \cite[Proposition 2.1]{he1} and \cite[Proposition 2.1]{he2}):
\begin{theorem}\label{HerzlichPMT-n}
Let $(M^n,g)$ be a $n$-dimensional Riemannian spin asymptotically flat manifold with integrable scalar curvature $R\geq 0$ and with a compact inner boundary $\partial M^n:=\coprod_{i=1}^k\Sigma_i$ such that 
\begin{eqnarray}\label{DiracMeanCur-n}
\lambda_1(\D_i)\geq\frac{1}{2}H_i>0
\end{eqnarray}
for all $i=1,...,k$. Then the mass is nonnegative and if the mass is zero, $(M^n,g)$ is flat, the mean curvature $H_i$ is constant and (\ref{DiracMeanCur-n}) is an equality for all $i=1,...,k$. 
\end{theorem}
Here $\lambda_1(\D_i)$ denotes the first nonzero eigenvalue of the extrinsic Dirac operator on $\Sigma_i$ endowed with the metric $\gb_i:=g_{|\Sigma_i}$. It is important to point out that this result is sharp since the exterior of round balls in Euclidean space are flat manifolds with zero mass for which (\ref{DiracMeanCur-n}) is an equality.


\subsection{First eigenvalue of the Dirac operator}


Our approach (following \cite{Simon}, see Appendix \ref{Simon}) consists to apply the previous positive mass theorem to a certain conformal deformation of the Riemannian manifold $(M^n,g)$. In particular, we have to check that the assumption (\ref{DiracMeanCur-n}) holds and so we have to deal with estimates for the first eigenvalue of the Dirac operator on compact Riemannian spin manifolds. This is a vast subject on which the interested reader may consult \cite{BHMM,G}. 

\subsubsection{The Friedrich inequality}

The first sharp inequality concerning eigenvalues of the Dirac operator on compact Riemannian spin manifolds is due to Friedrich \cite{F} and is now known as the Friedrich inequality. It asserts that, if $(\Sigma^{n-1},\gb)$ is such a manifold and if $\lambda_1(\D)$ denotes it first Dirac eigenvalue, then   
\begin{eqnarray}\label{Friedrich}
\lambda_1(\D)^2\geq\frac{n-1}{4(n-2)}\inf_{\Sigma^{n-1}}\Ss
\end{eqnarray} 
where $\Ss$ is the scalar curvature of $\Sigma^{n-1}$ with respect to the metric $\gb$. Moreover, equality occurs if and only if the manifold carries a real Killing spinor. In particular, it is an Einstein manifold with positive scalar curvature. Note that if $\Sigma^{n-1}$ is disconnected, this inequality holds on each of its connected components.

\subsubsection{A Hijazi-Montiel-Zhang-like inequality}\label{HMZsection}

If now we assume that $\Sigma^{n-1}$ is the boundary of an $n$-dimensional {\it compact} Riemannian spin manifold $(M^n,g)$, one can relate this first eigenvalue with extrinsic geometric invariants. In \cite{HMZ}, Hijazi, Montiel and Zhang prove that if the scalar curvature of $(M^n,g)$ and the mean curvature of $\Sigma^{n-1}=\partial M^n$ are both nonnegative, it holds that
\begin{eqnarray}\label{HMZ}
\lambda_1(\D_i)\geq\frac{1}{2}\inf_{\Sigma_i} H_i
\end{eqnarray}
for all $i=1,...,k$. Here we let $\Sigma^{n-1}:=\coprod_{i=1}^k\Sigma_i$ where for $i=1,...,k$, $\Sigma_i$ is a connected component of $\Sigma^{n-1}$ with metric $\gb_i:=g_{|\Sigma_i}$ and first Dirac eigenvalue $\lambda_1(\D_i)$. It turns out that, in our situation, this inequality does not apply directly since one cannot ensure, a priori, that the mean curvature is nonnegative. This is the first thing which has to be taken into account in our approach. The second one is that the metric fails to be smooth at an interior point (where it is in fact $C^{1,1}$). We shall see that the works of Bartnik and Chru\'sciel \cite{bc,bca} allow to deal with this problem. So we end this section by recalling some of their results and we refer to their papers for the definition of the Sobolev spaces which appear below. 

Consider $M^n$ a smooth, compact, spin manifold with boundary $\Sigma^{n-1}$ endowed with a Riemannian metric $g$ which is smooth on $M^n\setminus\{p\}$ and $W^{2,q}_{loc}$ at $p\in M^n\setminus\Sigma^{n-1}$ with $q>n/2$. The choice of the boundary condition for the Dirac operator in our approach is crucial and the motivations for this choice are similar to those of Herzlich (as explained in the previous section). The Atiyah-Patodi-Singer boundary condition $\mathcal{P}_{>0}$ is defined as the $L^2$-orthogonal projection on the positive eigenspaces of the Dirac operator $\D$. It is shown in \cite[Corollary 7.4]{bc} that the boundary problem
\begin{equation}\label{GeneralBVP}
\left\lbrace
\begin{array}{ll}
D\psi=\eta & \text{ on } M^n\\
\mathcal{P}_{>0}\psi_{|\Sigma^{n-1}}=\zeta & \text{ along }\Sigma^{n-1}
\end{array}
\right.
\end{equation}
for $(\eta,\zeta)\in L^2\times \mathcal{P}_{>0}H_*^{1/2}$ has a solution $\psi\in H^1$ if and only if ${\rm Ker\,}(D^*,\mathcal{P}^*_{>0})$ is reduced to zero. Here $D^*$ denotes the $L^2$-formal adjoint of $D$ and $\mathcal{P}^*_{>0}$ the adjoint boundary condition of $\mathcal{P}_{>0}$ . It is in fact straightforward to see that $D^*=D$ and $\mathcal{P}^*_{>0}=\mathcal{P}_{\geq 0}$, the $L^2$-orthogonal projection on the nonnegative eigenspaces of $\D$, so that the condition on the existence of a solution to (\ref{GeneralBVP}) can be expressed as
\begin{eqnarray}\label{CritereSol}
{\rm Ker\,}(D,\mathcal{P}_{\geq 0})=\{0\}.
\end{eqnarray}
On the other hand, if the data $\Psi$ and $\Phi$ are smooth, the interior \cite[Theorem 3.8]{bc} and boundary \cite[Theorem 6.6]{bc} regularity results apply on $M^n\setminus\{p\}$, since the metric is smooth there. Therefore, the spinor field $\psi$ is smooth on $M^n\setminus\{p\}$. As we will see in Section \ref{MainLemmaProof}, these facts imply in particular that the Hijazi-Montiel-Zhang inequality (\ref{HMZ}) holds under these weaker assumptions on $(M^n,g)$.


\section{Proof of Theorem \ref{MainTheorem}}\label{ProofMainTheorem}


Let $(M^n,g,N)$ be a spin asymptotically isotropic pseudo-static system of mass $m$ with $n\geq 3$. In the following, we write $\partial M^n=\coprod_{i=1}^k\Sigma_i$ and we assume that $\Sigma_i$ is
\begin{itemize}
\item a nondegenerate static horizon for $1\leq i\leq i_0$,
\item a generalized quasilocal photon surface for $i_0+1\leq i\leq k$.
\end{itemize}

Then consider the metric conformally related to $g$ defined by
\begin{eqnarray}\label{g-plus}
g^+=\Phi_+^{\frac{4}{n-2}}g\quad\text{with}\quad\Phi_+=\frac{1+N}{2}.
\end{eqnarray}
From the well-known transformation of the scalar curvature under a conformal change of the metric, the scalar curvature $R^+$ of $g^+$ is easily computed to be
\begin{eqnarray*}
R^{+}=\Phi_+^{-\frac{n+2}{n-2}}\Big(-4\frac{n-1}{n-2}\Delta\Phi_++R\Phi_+\Big)=R\Phi_+^{-\frac{4}{n-2}}\geq 0
\end{eqnarray*}
since $N$ is harmonic with respect to $g$ and $R\geq 0$. On the other hand, a direct computation using the asymptotically isotropy of the metric $g$ as well as the one of the lapse $N$ allows to prove that $(M^n,g^+)$ is asymptotically flat with zero ADM mass. It is then enough to check that the condition (\ref{DiracMeanCur-n}) in Theorem \ref{HerzlichPMT-n} is fulfilled to conclude that $g^+$ is flat. For this reason, one has to compute the mean curvature of each boundary components with respect to the metric $g^+$. This is achieved by using the classical formula which relates the mean curvature of the boundary of two conformally related metrics, namely
\begin{eqnarray*}
H_{i}^+=\Phi_{+}^{-\frac{n}{n-2}}\Big(2\frac{n-1}{n-2}\frac{\partial \Phi_+}{\partial\nu}+H_i\Phi_+\Big)
\end{eqnarray*}
for all $i\in\{1,...,k\}$. Then if $\Sigma_i$ is a nondegenerate static horizon, it follows from (\ref{SemiStatic}) and the previous formula that its mean curvature is given by the positive constant
\begin{eqnarray}\label{HplusStatic}
H^+_{i}=2^{\frac{n}{n-2}}\frac{n-1}{n-2}\kappa_i
\end{eqnarray}
where $\kappa_i:=\nu_i(N)>0$ is the surface gravity of $\Sigma_i$ for $i\in\{1,...,i_0\}$. On the other hand, it follows from (\ref{QuasiLocalPhotonSpheres2}) that if $i\in\{i_0+1,...,k\}$, the mean curvature of $\Sigma_i$ is also a positive constant whose value is
\begin{eqnarray}\label{HplusPhoton}
H^+_i=\frac{1}{2}H_i\Phi_{+}^{-\frac{n}{n-2}}\Big(c_i N_i+1\Big).
\end{eqnarray}
A first attempt to prove that the condition (\ref{DiracMeanCur-n}) holds is to apply the Friedrich inequality (\ref{Friedrich}) on each component of the boundary. It is easily seen to be unfruitful for the nondegenerate static horizon components. For the generalized quasilocal photon surfaces, it is useful to consider the integer $j_0\in\{i_0+1,...,k\}$ for which
\begin{eqnarray*}
N_i \geq c_{i}^{-1/2}
\end{eqnarray*}
if $i_0+1\leq i\leq j_0$ and 
\begin{eqnarray*}
N_i \leq c_{i}^{-1/2}
\end{eqnarray*} 
if $j_0+1\leq i\leq k$. Then we compute that the scalar curvature $\Ss^+_i$ of $(\Sigma_i,\gb^+_i)$ with $\gb^+_i:=g_{|\Sigma_i}^+$ satisfies
\begin{eqnarray}\label{ScalarBoundaryPlus}
\Ss^+_i=\Phi_{+}^{-\frac{4}{n-2}}\Ss_i\geq\frac{n-2}{n-1}\Phi_+^{-\frac{4}{n-2}}c H^2_i>0
\end{eqnarray} 
for $i\in\{i_0+1,...,k\}$. The previous inequality is a direct consequence of (\ref{GeneralizedQPS}). Combining the Friedrich inequality (\ref{Friedrich}) for the first eigenvalue $\lambda_1(\D^+_i)$ of the Dirac operator $\D^+_i$ on $(\Sigma_i,\gb^+_i)$ with (\ref{ScalarBoundaryPlus}) yields  
\begin{eqnarray*}
\lambda_1(\D^+_i)^2\geq\frac{n-1}{4(n-2)}\inf_{\Sigma_i}\Ss^{+}_i\geq\frac{1}{4}\Phi_+^{-\frac{4}{n-2}}c_i H_i^2.
\end{eqnarray*}
The last inequality is deduced from the facts that the function $\Phi_+$ and the mean curvature $H_i$ are constant on the generalized quasilocal photon surface $\Sigma_i$. Now from (\ref{HplusPhoton}), we have that
\begin{eqnarray}\label{FriedComp}
\frac{1}{4}\Phi_+^{-\frac{4}{n-2}}c_i H^2_i\geq\frac{1}{4}\big(H^+_i\big)^2
\end{eqnarray}
if and only if
\begin{eqnarray*}
\frac{1}{4}\Phi_+^{-\frac{4}{n-2}}c_i H^2_i\geq\frac{1}{16}H^2_i\Phi_+^{-\frac{2n}{n-2}}\Big(c_i N_i+1\Big)^2 
\end{eqnarray*}
that is 
\begin{eqnarray*}
c_i\Big(N_i+1\Big)^2\geq\Big(c_iN_i+1\Big)^2.
\end{eqnarray*} 
However, since $c_i>1$, it is easy to observe that this inequality holds only for $i\in\{j_0+1,...,k\}$. It remains to show that (\ref{DiracMeanCur-n}) is also true for $i\in\{1,...,j_0\}$. For this, consider the smooth metric defined on $M^n$ by 
\begin{eqnarray}\label{g-moins}
g^-=\Phi_-^{\frac{4}{n-2}}g\quad\text{with}\quad\Phi_-=\frac{1-N}{2}
\end{eqnarray}
which is Riemannian since $N$ is harmonic and asymptotically isotropic. Similarly to the metric $g^+$, the scalar curvature $R^-$ of $g^-$ is nonnegative. Moreover since $g$ is asymptotically isotropic, it can be shown (\cite{C2,C3,J}) that one can insert a point $p_\infty$ into $(M^n,g^-)$ to obtain a compact Riemannian spin manifold $(M^n_\infty:=M^n\cup\{p_\infty\},g^-_\infty)$ whose metric is smooth on $M^n$, $C^{1,1}$ at $p_\infty$ and with boundary $\partial M^n_\infty=\partial M^n$. Then the mean curvature of $\partial M^n$ in $(M^n_\infty,g^-_\infty)$ computed with respect to the unit normal 
\begin{eqnarray*}
\nu_-=-\Phi_-^{-\frac{2}{n-2}}\nu
\end{eqnarray*}
is 
\begin{eqnarray*}
H^{-}_i=-\Phi_-^{-\frac{n}{n-2}}\Big(2\frac{n-1}{n-2}\frac{\partial \Phi_-}{\partial\nu}+H_i\Phi_-\Big).
\end{eqnarray*}
For $i=1,...,i_0$, we have that  
\begin{eqnarray}\label{Hmoins}
H_i^-=2^{\frac{n}{n-2}}\frac{n-1}{n-2}\kappa_i>0
\end{eqnarray}
is a positive constant, once again because of (\ref{SemiStatic}). On the other hand, if $i\in\{i_0+1,...,k\}$, we deduce using (\ref{QuasiLocalPhotonSpheres2}) that
\begin{eqnarray}\label{Hmoins1}
H^{-}_i=\frac{1}{2}H_i\Phi_-^{-\frac{n}{n-2}}\Big(c_i N_i-1\Big)
\end{eqnarray}
is also a constant but it could, a priori, be negative for $i\in\{j_0+1,...,k\}$. This is the main reason why we cannot apply directly the inequality (\ref{HMZ}). Instead, we use the following result whose proof is postponed to the next section. 
\begin{lemma}\label{MainLemma}
For $i\in\{1,...,j_0\}$, the first eigenvalue $\lambda_1(\D^-_i)$ of the Dirac operator $\D^-_i$ on $(\Sigma_i,\gb^-_i)$ satisfies 
\begin{eqnarray*}
\lambda_1(\D^-_i)\geq\frac{1}{2}H^-_i.
\end{eqnarray*}
Moreover, if equality holds, the boundary $\partial M^n$ is connected. 
\end{lemma}
As $H^-_i=H^+_i$ for all $i\in\{1,...,i_0\}$ because of (\ref{HplusStatic}) and (\ref{Hmoins}), Lemma \ref{MainLemma} allows to deduce directly that (\ref{DiracMeanCur-n}) holds for nondegenerate static horizons since $\gb^-_i=\gb^+_i$ implies $\lambda_1(\D^-_i)=\lambda_1(\D^+_i)$. If $i\in\{i_0+1,...,j_0\}$, we remark that since $\Phi_+$ and $\Phi_-$ are constant on $\Sigma_i$ and 
\begin{eqnarray*}
\gb^+_i=\Big(\frac{\Phi_+}{\Phi_-}\Big)^{\frac{4}{n-2}}\gb^-_i,
\end{eqnarray*}
the corresponding Dirac operators are related by 
\begin{eqnarray*}
\D^+_i=\Big(\frac{\Phi_-}{\Phi_+}\Big)^{\frac{2}{n-2}}\D^-_i
\end{eqnarray*}
so that the corresponding first eigenvalues satisfy
\begin{eqnarray}\label{HomotheticEigen}
\lambda_1(\D^+_i)=\Big(\frac{\Phi_-}{\Phi_+}\Big)^{\frac{2}{n-2}}\lambda_1(\D^-_i).
\end{eqnarray}
Therefore, the inequality of Lemma \ref{MainLemma} reads
\begin{eqnarray*}
\lambda_1(\D^+_i)\geq\frac{1}{2}\Big(\frac{\Phi_-}{\Phi_+}\Big)^{\frac{2}{n-2}}H^-_i.
\end{eqnarray*}
Now, from (\ref{HplusPhoton}) and (\ref{Hmoins1}), we remark that 
\begin{eqnarray*}
\Big(\frac{\Phi_-}{\Phi_+}\Big)^{\frac{2}{n-2}}H^-_i\geq H^+_i
\end{eqnarray*}
if and only if
\begin{eqnarray*}
\Big(c N_i-1\Big)\Phi_+\geq\Big(c N_i+1\Big)\Phi_-
\end{eqnarray*}
which is true precisely if $i\in\{i_0+1,...,j_0\}$. In conclusion, the assumption (\ref{DiracMeanCur-n}) of the Positive Mass Theorem \ref{HerzlichPMT-n} holds for every component of the boundary, so that the asymptotically flat manifold $(M^n,g^+)$ with nonnegative scalar curvature and zero ADM mass has to be flat. Moreover, each $\Sigma_i$ is totally umbilical with constant mean curvature. 

If we chose a nondegenerate static horizon $\Sigma_l$ with $l\in\{1,...,i_0\}$, since equality holds in (\ref{DiracMeanCur-n}), we also have equality in Lemma \ref{MainLemma} and so $\partial M^n=\Sigma_l$ is connected. The Gauss formula relative to the immersion of $\partial M^n$ in $(M^n,g^+)$ implies that $\partial M^n$ has positive constant sectional curvature and it has to be isometric to the quotient of a round sphere. Arguing as in \cite{he2} we conclude that $(M^n,g^+)$ is (up to an isometry) the exterior of a round ball with radius
\begin{eqnarray*}
s_l=\frac{n-2}{2^{\frac{n}{n-2}}\kappa_l}.
\end{eqnarray*} 
The metric $g^+$ being the Euclidean one, it turns out that the function $\Psi:=\Phi_+^{-1}$ satisfies the boundary value problem
$$
\left\lbrace
\begin{array}{ll}
\Delta_\delta \Psi= 0 & \text{ on } \mathbb{R}^n\setminus\overline{B}_{s_l}(0)\\
\Psi(s) =1+O(s^{2-n})& \text{ as }s\rightarrow\infty\\
\Psi_{|\partial B_{s_l}(0)}=2 & 
\end{array}
\right.
$$
where $\Delta_\delta$ denotes the Euclidean Laplacian. The maximum principle implies the uniqueness of solutions to this elliptic equation so that it is straightforward to check that 
\begin{eqnarray*}
\Psi(s)=1+\Big(\frac{s_l}{s}\Big)^{n-2}
\end{eqnarray*} 
is the only one. The Riemannian manifold $(M^n,g)$ has to be isometric to the exterior of the Schwarzschild metric with mass 
\begin{eqnarray*}
m=2s_l^{n-2}=\frac{1}{2}\Big(\frac{A}{\omega_{n-1}}\Big)^{\frac{n-2}{n-1}}>0
\end{eqnarray*} 
and $N=N_m$. Here $A$ denotes the area of the inner boundary $(\partial M^n,\gb)$. 
 
Now assume that we chose a quasilocal photon surface $\Sigma_l$ with $l\in\{i_0+1,...,j_0\}$. As before, since  (\ref{DiracMeanCur-n}) is an equality, a straightforward computation implies that equality also holds in Lemma \ref{MainLemma} and the boundary $\partial M^n$ has to be connected, that is $\partial M^n=\Sigma_l$. Moreover, since the boundary is totally umbilical with constant mean curvature, one can argue as in the previous paragraph to conclude that $(M^n,g^+)$ is isometric to the exterior of an Euclidean ball with radius 
\begin{eqnarray*}
\widetilde{s}_l=\frac{(n-1)2^{-\frac{2}{n-2}}c_l^{-\frac{n}{2(n-2)}}\big(c_l^{1/2}+1\big)^{\frac{2}{n-2}}}{H_l}
\end{eqnarray*}
since $N_l=c_l^{-1/2}$. Once again, writing $g=\Psi^{\frac{4}{n-2}}\delta$ with $\Psi:=\Phi_+^{-1}$ and $\delta$ the Euclidean metric, we deduce that $\Psi$ is the unique solution of the problem 
$$
\left\lbrace
\begin{array}{ll}
\Delta_\delta \Psi= 0 & \text{ on } \mathbb{R}^n\setminus\overline{B}_{\widetilde{s}_l}(0)\\
\Psi(s) =1+O(s^{2-n})& \text{ as }s\rightarrow\infty\\
\Psi_{|\partial B_{\widetilde{s}_l}(0)}=2(1+c_l^{-1/2})^{-1}. & 
\end{array}
\right.
$$
It is then straightforward to check that 
\begin{eqnarray*}
\Psi(s)=1+\Big(\frac{\widehat{s}_l}{s}\Big)^{n-2}
\end{eqnarray*} 
with 
\begin{eqnarray*}
\widehat{s}_l=\frac{(n-1)2^{-\frac{2}{n-2}}c_l^{-\frac{n}{2(n-2)}}\big(c_l-1\big)^{\frac{1}{n-2}}}{H_l}.
\end{eqnarray*}
We conclude that $(M^n,g)$ is isometric to the exterior $\{s\geq\widetilde{s}_l\}$ of the Schwarzschild metric with mass 
\begin{eqnarray*}
m=2\widehat{s}_l^{\,n-2}=\frac{1}{2}\Big(1-\frac{1}{c_l}\Big)\Big(\frac{A}{\omega_{n-1}}\Big)^{\frac{n-2}{n-1}}>0
\end{eqnarray*}
and $N=N_m$. Here we used the Stokes' theorem and the harmonicity of $N$ to compute that
\begin{eqnarray}\label{MassFormula}
\int_{\partial M^n}\nu(N)=(n-2)\omega_{n-1} m
\end{eqnarray}
for any asymptotically isotropic pseudo-static systems of mass $m$ with connected inner boundary. Then, using (\ref{QuasiLocalPhotonSpheres2}) and the fact that $N_l=1/\sqrt{c_l}$, yields 
\begin{eqnarray*}
H_l=2(n-1)\frac{\omega_{n-1}}{A}\frac{\sqrt{c_l}}{c_l-1}m
\end{eqnarray*}
and the last expression of the mass follows.

Finally, if we chose a quasilocal photon surface $\Sigma_l$ with $l\in\{j_0+1,...,k\}$ then, since (\ref{DiracMeanCur-n}) is an equality, we directly get that the equality is also achieved in (\ref{FriedComp}). This implies that $N_i=c_i^{-1/2}$ for all $i\in\{j_0+1,...,k\}$ and thus from (\ref{Hmoins1}) that $H^-_i$ is a positive constant. Now since {\it all} the components of the boundary have positive mean curvature, the Hijazi-Montiel-Zhang inequality (\ref{HMZ}) applies (see Section \ref{MainLemmaProof}) and it is in fact an equality. The boundary has therefore to be connected and totally umbilical with constant mean curvature. One can conclude exactly as in the previous situation  and this finish the proof of Theorem \ref{MainTheorem}. 


\section{Proof of Lemma \ref{MainLemma}}\label{MainLemmaProof}


In order to prove Lemma \ref{MainLemma}, we first need to show that the boundary value problem (\ref{GeneralBVP}) for the Dirac operator $D^-$ on $(M^n_\infty,g^-_\infty)$ under the Atiyah-Patodi-Singer condition $\mathcal{P}_{>0}^-$ admits an unique strong solution. Here $\mathcal{P}_{>0}^-$ denotes the $L^2$-orthogonal projection on the positive eigenspaces of the boundary Dirac operator $\D^-$. Since the metric $g^-_\infty$ is smooth on $M^n$ and $W^{2,q}_{loc}$ for $q>n/2$ at the interior point $p_\infty$ (since $C^{1,1}$ at this point), the results of Bartnik and Chru\'sciel recalled in Section \ref{HMZsection} apply on $(M^n_\infty,g^-_\infty)$. So proving the existence of such a solution is reduced to show that (\ref{CritereSol}) holds under the assumptions of Lemma \ref{MainLemma}. 

For this, we recall the integral version of the famous Schr\"odinger-Li\-chne\-ro\-wicz formula (see \cite{bc,HMZ} for a proof) which states that
\begin{eqnarray}\label{Integral-SL}
\int_{M^n_\infty}\Big(|\nabla^-\varphi|^2+\frac{R^-}{4}|\varphi|^2-|D^-\varphi|^2\Big)=\sum_{i=1}^{k}\int_{\Sigma_i}\Big(\<\D^-_i\varphi_i,\varphi_i\>-\frac{H^-_i}{2}|\varphi_i|^2\Big)
\end{eqnarray}
for all $\varphi\in\Gamma(\mathbf{S}^-)$ and where $\varphi_i:=\varphi_{|\Sigma_i}\in\Gamma(\RSB^-_i)$. The notations in this formula are directly derived from those of Section \ref{SpinTools}. 

Consider now $\psi\in{\rm Ker\,}(D^-,\mathcal{P}^-_{\geq 0})$ that is $\psi\in H^1$ and satisfies
\begin{equation}\label{CokernelBC}
\left\lbrace
\begin{array}{ll}
D^-\psi=0 & \text{ on } M^n_\infty\\
\mathcal{P}^-_{\geq 0}\psi_i=0 & \text{ along }\Sigma_i 
\end{array}
\right.
\end{equation}
for all $i\in\{1,...,k\}$. Since the scalar curvature of the metric $g^-_\infty$ is nonnegative, the formula (\ref{Integral-SL}) applied to $\psi$ yields
\begin{eqnarray}\label{CokernelBound-1}
\sum_{i=1}^{k}\int_{\Sigma_i}\Big(\<\D^-_i\psi_i,\psi_i\>-\frac{H^-_i}{2}|\psi_i|^2\Big)\geq 0.
\end{eqnarray}
On the other hand, it is straightforward to check that for all $\varphi\in\Gamma(\mathbf{S}^-)$, we have
\begin{eqnarray}\label{BoundAPS-1}
\int_{\Sigma_i}\<\D^-_i\varphi_i,\varphi_i\>\leq \int_{\Sigma_i}\<\D^-_i\big(\APSm\varphi_i\big),\APSm\varphi_i\>
\end{eqnarray}
with equality if and only if the $L^2$-projection of $\varphi_i$ on the negative eigenspaces of $\D^-_i$ is zero. So we can rewrite (\ref{CokernelBound-1}) as
\begin{eqnarray}\label{CokernelBound-2}
\sum_{i=j_0+1}^{k}\int_{\Sigma_i}\Big(\<\D^-_i\psi_i,\psi_i\>-\frac{H^-_i}{2}|\psi_i|^2\Big)\geq 0
\end{eqnarray}
because of the boundary conditions in (\ref{CokernelBC}) and since $H_i^->0$ for $i\in\{1,...,j_0\}$. Note that we didn't use (\ref{BoundAPS-1}) directly for $i\in\{j_0+1,...,k\}$, since the mean curvature $H^-_i$ may be negative and the inequality would then be useless. Instead, we remark that since the $L^2$-projection of $\psi_i$ on the kernel of $\D_i^-$ is zero, the upper bound (\ref{BoundAPS-1}) can be refined to
\begin{eqnarray}\label{BoundAPS-2}
\int_{\Sigma_i}\<\D^-_i\psi_i,\psi_i\>\leq -\lambda_1(\D^-_i)\int_{\Sigma_i}|\psi_i|^2
\end{eqnarray}
for $i\in\{j_0+1,...,k\}$ with equality if and only if $\psi_i$ is an eigenspinor for $\D^-_i$ associated to the eigenvalue $-\lambda_1(\D^-_i)$ or is zero on $\Sigma_i$. This allows to rewrite (\ref{CokernelBound-1}) as
\begin{eqnarray}\label{CokernelBound-3}
\sum_{i=j_0+1}^{k}\Big(\lambda_1(\D^-_i)+\frac{H^-_i}{2}\Big)\int_{\Sigma_i}|\psi_i|^2\leq 0.
\end{eqnarray}
Assume for a moment that 
\begin{eqnarray}\label{PositiveConstants}
\lambda_1(\D^-_i)+\frac{H^-_i}{2}>0
\end{eqnarray}
for all $i\in\{j_0+1,...,k\}$. We thus have equality in (\ref{CokernelBound-3}) and so $\psi_i=0$ for all $i\in\{j_0+1,...,k\}$. Moreover, equality also holds in (\ref{CokernelBound-2}) and so the $L^2$-projection of $\psi_i$ on the negative eigenspaces of $\D^-_i$ vanishes for all $i\in\{1,...,j_0\}$. But $\psi_i$ satisfies the boundary condition in (\ref{CokernelBC}) and then we deduce that $\psi_i=0$ for all $i\in\{1,...,k\}$. Finally, equality occurs in (\ref{Integral-SL}) and so we get a parallel spinor field $\psi$ which vanishes along $\partial M^n$. Since $\psi$ is smooth on $M^n$, its norm is constant on $M^n$. But since it is zero on the boundary, it has to be zero on the whole of $M^n$ and so (\ref{CritereSol}) is fulfilled. It remains to show that (\ref{PositiveConstants}) is satisfied. For this, we put together (\ref{DiracMeanCur-n}) (which holds for $i=j_0+1,...,k$ without using Lemma \ref{MainLemma} by the Friedrich inequality) and the formula (\ref{HomotheticEigen}) which relates the first eigenvalue of Dirac operators for homothetic metrics to get 
\begin{eqnarray*}
\lambda_1(\D^-_i)\geq\frac{1}{2}\Big(\frac{\Phi_+}{\Phi_-}\Big)^{\frac{2}{n-2}}H_i^+.
\end{eqnarray*}
This lower bound and the expressions (\ref{HplusPhoton}) and (\ref{Hmoins1}) of the mean curvatures of $\Sigma_i$ with respect to the metrics $g^+$ and $g^-$ in terms of the metric $g$ yield
\begin{eqnarray*}
\lambda_1(\D^-_i)+\frac{H^-_i}{2}\geq
\frac{1}{2}\frac{H_iN_i}{1+N_i}\Phi_{-}^{-\frac{n}{n-2}}(c_i-1).
\end{eqnarray*}
This implies (\ref{PositiveConstants}) since the right-hand side of the previous inequality is clearly positive because $H_i$ as well as $N_i$ are positive constants and $c_i>1$ for all $i\in\{j_0+1,...,k\}$. 

Now the previous discussion ensures that for a fixed $l\in\{1,...,j_0\}$, the boundary value problem
$$
\left\lbrace
\begin{array}{ll}
D^-\xi=0 & \text{ on } M^n_\infty\\
\mathcal{P}^-_{>0}\xi_{|\partial M^n}=\eta_l\ & \text{ along }\partial M^n 
\end{array}
\right.
$$
admits a solution $\xi\in\Gamma(\mathbf{S}^-)$, where $\eta_l\in\Gamma(\RSB^-)$ is defined by
$$\eta_l=
\left\lbrace
\begin{array}{ll}
\zeta_l & \text{ on } \Sigma_l\\
0 & \text{ on }\Sigma_i\text{ for } i\neq l
\end{array}
\right.$$
with $\zeta_l\in\Gamma(\RSB^-_l)$ is a smooth eigenspinor for $\D^-_l$ associated to $\lambda_1(\D^-_l)$. Note that as recalled in Section \ref{HMZsection}, $\xi$ is smooth on $M^n$. Applying the integral version of the Schr\"odinger-Lichnerowicz formula (\ref{Integral-SL}), the upper bounds (\ref{BoundAPS-1}) and (\ref{BoundAPS-2}) respectively for $i\in\{1,...,j_0\}$ and for $i\in\{j_0+1,...,k\}$ to the spinor field $\xi$ give
\begin{eqnarray*}
\Big(\lambda_1(\D^-_l)-\frac{H^-_l}{2}\Big)\int_{\Sigma_l}|\xi_l|^2 & \geq & \sum_{i=j_0+1}^{k}\int_{\Sigma_i}\Big(\frac{H^-_i}{2}|\xi_i|^2-\<\D^-\xi_i,\xi_i\>\Big)\\
& \geq & \sum_{i=j_0+1}^{k}\Big(\lambda_1(\D^-_i)+\frac{H^-_i}{2}\Big)\int_{\Sigma_i}|\xi_i|^2.
\end{eqnarray*}
From (\ref{PositiveConstants}), we immediately obtain the inequality of Lemma \ref{MainLemma}. We should pay attention to the fact that the upper bound (\ref{BoundAPS-2}) holds for spinor fields $\psi_i\in\Gamma(\RSB^-_i)$ such that $\mathcal{P}^-_{\geq 0}\psi_i=0$ and so we have to check that $\xi$ satisfied this property. Since $\mathcal{P}^-_{>0}\xi_i=0$ for all $i\neq l$, it is enough to show that the projection of $\xi_i$ on the kernel of the Dirac operator $\D_i^-$ is reduced to zero for $i\in\{j_0+1,...,k\}$. In fact, for such $i$, the Friedrich inequality (\ref{Friedrich}) combined with the positivity (\ref{GeneralizedQPS}) of the scalar curvature of $(\Sigma_i,\gb_i)$ ensures that the kernel of $\D_i$ is reduced to zero. But the dimension of this space being invariant under conformal changes of the metric on $\Sigma_i$ (see \cite{Hi}), we deduce that the kernel of $\D^-_i$ is also reduced to zero and so 
$$\mathcal{P}^-_{\geq 0}\xi_i=\mathcal{P}^-_{>0}\xi_i=0.$$
Now if equality occurs in Lemma \ref{MainLemma}, it it immediate to see that equality also holds in (\ref{Integral-SL}) and the spinor field $\xi$ has to be parallel. On the other hand, since $H^-_i>0$ for all $i\in\{1,...,j_0\}$, the kernel of $D^-_i$ is also reduced to zero and so ${\rm Ker\,}(\D^-)=\{0\}$, where $\D^-$ is the full boundary Dirac operator of $(\partial M^n,\gb^-)$. This allows to conclude that 
\begin{eqnarray*}
\xi_{|\partial M^n}=\mathcal{P}^-_{>0}\eta_{l}=\eta_l
\end{eqnarray*} 
since equality has to hold in (\ref{BoundAPS-1}) and (\ref{BoundAPS-2}). As the squared norm of the parallel spinor $\xi$ is smooth on $M^n$, it has to be a positive constant because its restriction to $\Sigma_l$ is a nonzero eigenspinor for $\mathcal{D}^-_l$. However, this is impossible if $\partial M^n$ is disconnected since otherwise it is also zero on the other connected components. We conclude that the boundary $\partial M^n$ is connected and the proof of Lemma \ref{MainLemma} is now complete.


\appendix
\section{Simon's proof of the static black hole uniqueness theorem of \cite{MZH}}\label{Simon}


In this section, we give the proof of the $(3+1)$-dimensional static black hole uniqueness theorem in the connected boundary case by Walter Simon \cite{Simon} which, in our conventions and notations, can be stated as follows:
\begin{theorem}\label{BHU}
An asymptotically flat static vacuum triple $(M^3,g,N)$ with a nondegenerate connected and compact inner boundary $\partial M^3:=N^{-1}(\{0\})$ is isometric to the exterior $\{s\geq (m/2)^{n-2}\}$ of the Schwarzschild manifold of positive mass $m$. 
\end{theorem}

Here the $3$-dimensional static vacuum triple $(M^3,g,N)$ is asymptotically flat in the sense that the manifold $M^3$ is diffeomorphic to the union of a compact set and an open end $E^3$ which is diffeomorphic to $\mathbb{R}^3\setminus \overline{B}$ where $B$ is an open ball in $\mathbb{R}^3$. Furthermore, we require that, with respect to the coordinates induced by this diffeomorphism, the metric $g$ and the lapse function $N$ satisfy
\begin{equation}\label{AsympFlat}
\begin{array}{l}
g_{ij}-\delta_{ij}\in W^{k,q}_{-\tau}(E)\\
N-1\in W^{k+1,q}_{-\tau}(E)
\end{array}
\end{equation}
for some $\tau>1/2$, $\tau\notin \mathbb{Z}$, $k\geq 2$, $q>4$. Furthermore, nondegenerate mean here that the surface gravity $\kappa=\nu(N)$ of the inner boundary is nonzero. Then, since $(M^3,g,N)$ is a static vacuum triple, it is a well known fact that this implies that the inner boundary is a nondegenerate static horizon in the sense of Section \ref{Setting}.

One of the idea of Simon was to apply the positive mass theorem with boundary of Herzlich to the metric $g^+$ defined by (\ref{g-plus}). As done before, it is enough to show that (\ref{DiracMeanCur-n}) holds in this situation and one way to obtain such a lower bound is to use the Friedrich inequality (\ref{Friedrich}). Then we are brought to compare the scalar curvature $\Ss^+$ of $(\partial M^3,\gb^+)$ with the mean curvature $\partial M^3$ in $(M^3,\gb^+)$ whose value is 
\begin{eqnarray}\label{Hplus-3}
H^+=16\kappa
\end{eqnarray}
as computed in (\ref{HplusStatic}). This can be achieved by using a divergence identity regarding the static vacuum triple in the compactified metric $g^-_\infty$ defined in Section \ref{ProofMainTheorem}. It turns out from the works of Beig and Simon \cite{BS1,BS2} and Kennefick and \'O Murchadha \cite{KOM} that since $(M^3,g,N)$ is an asymptotically flat static vacuum triple with non-zero mass, its one-point compactification $(M^3_\infty,g_\infty^-)$ admits a smooth (even analytic) extension to $p_\infty$. The fact that the mass is non-zero in our situation follows directly by formula (\ref{MassFormula}). Then from the divergence identity (see \cite{Simon2} and Remark \ref{DivIdenRem} for a generalization of this formula)
\begin{eqnarray}\label{DivIden}
{\rm div}_{g^-}\Big(N^{-1}(1+N)^2\nabla^-W\Big)=\frac{1}{8}V|Ric^-|_{g^-}^2
\end{eqnarray}
where $W$ and $V$ are the respectively smooth and continuous nonnegative functions on $M^3_\infty$ given by 
\begin{eqnarray*}
W=\frac{|\nabla N|^2}{(1-N^2)^4}\quad\text{and}\quad V=N\frac{(1-N)^2}{(1+N)^4},
\end{eqnarray*}
one can deduce the following lemma which plays a crucial role in what follows. Here $\nabla^-$, ${\rm div}_{g^-}$ and $Ric^-$ are respectively the covariant derivative, the divergence and the Ricci curvature of the manifold $M^3_\infty$ with respect to the metric $g^-_\infty$. 
\begin{lemma}\label{TechnicalLemma}
The scalar curvature $\Ss$ of the nondegenerate static horizon $\partial M^3$ with respect to the metric $g$ satisfies $\Ss\geq 8\kappa^2$.
\end{lemma}

\begin{proof}
Applying the weak maximum principle to (\ref{DivIden}) on the smooth compact manifold $M^3_\infty$, we find that $W$ is constant on $M^3_\infty$ or $W$ must take its maximum on the nondegenerate static horizon where it is constant. In either case, the derivative of $W$ along the inward normal $\nu$ is nonpositive near $\partial M^3_\infty$ and the same applies for the following limit
\begin{eqnarray}\label{LimitMax}
\lim_{N\rightarrow 0} \kappa^{-1} N^{-1} g(\nabla W,\nabla N)\leq 0.
\end{eqnarray}
On the other hand, a direct computation using the static equations (\ref{vacuum}) yields 
\begin{eqnarray*}
g(\nabla W,\nabla N) & = & (1-N^2)^{-4}\Big(\nabla N\big(|\nabla N|^2\big)+8N|\nabla N|^4(1-N^2)^{-1}\Big)\\
& = & 2(1-N^2)^{-4}N\Big(Ric(\nabla N,\nabla N)+4|\nabla N|^4(1-N^2)^{-1}\Big)\\
\end{eqnarray*}
which allows to rewrite (\ref{LimitMax}) as
\begin{eqnarray*}
0\geq 2\kappa\Big(Ric(\nu,\nu)+4\kappa^2\Big)=\kappa\Big(-\frac{\Ss}{2}+4\kappa^2\Big).
\end{eqnarray*}
Here we used the fact that $\nu=\kappa^{-1}\nabla N$ and, in the last equality, the Gauss formula (recall that $\partial M^3$ is totally geodesic in $(M^3,g)$ which is scalar flat). This conclude the proof of this lemma.
\end{proof}

Now, since the manifold $M^3$ is $3$-dimensional, it is automatically endowed with a spin structure so that the first eigenvalue of the Dirac operator on $\partial M^3$ with the metric $\gb^+$ satisfies the Friedrich inequality
\begin{eqnarray}\label{Fried2}
\lambda_1(\D^+)^2\geq\frac{1}{2}\inf_{\partial M^3}\Ss^+.
\end{eqnarray}
However, it is now immediate from the formula (\ref{ScalarBoundaryPlus}) that the Lemma \ref{TechnicalLemma} implies that 
\begin{eqnarray*}
\Ss^+=16 \Ss\geq 128 \kappa^2.
\end{eqnarray*}
Combining this estimate with the inequality (\ref{Fried2}) yields  
\begin{eqnarray*} 
\lambda_1(\D^+)^2\geq 64\kappa^2
\end{eqnarray*}
which, because of (\ref{Hplus-3}), is exactly (\ref{DiracMeanCur-n}). We conclude the proof of Theorem \ref{BHU} as usual. 

\begin{remark}\label{DivIdenRem}
It may be interesting to note that the divergence identity (\ref{DivIden}) can be generalized to $n$-dimensional static vacuum triples $(M^n,g,N)$ for the metrics $g^\pm$. Namely it holds that
\begin{eqnarray}\label{DivIdenRem-1}
{\rm div}_{g^\pm}\Big(N^{-1}(1\mp N)^2\nabla^\pm W\Big)=2^{\frac{n-6}{n-2}}V^\pm|Ric^\pm|_{g^\pm}^2
\end{eqnarray}
where $W$ and $V^\pm$ are the functions defined by 
\begin{eqnarray*}
W=\frac{|\nabla N|^2}{(1-N^2)^{2\frac{n-1}{n-2}}}\quad\text{and}\quad V^\pm=N\frac{(1\pm N)^{\frac{2}{n-2}}}{(1\mp N)^{2\frac{n-1}{n-2}}}.
\end{eqnarray*}
Indeed, from the second Bianchi identity and the fact that $g^\pm$ is scalar flat, one computes that 
\begin{eqnarray}\label{Pohozaev-Pointwise}
{\rm div}_{g^\pm}\Big(Ric^\pm(X)\Big)=\frac{1}{2}\<Ric^\pm,\mathcal{L}_Xg^\pm\>_{g^{\pm}}
\end{eqnarray}
for all $X\in\Gamma(TM^n)$ and where $\mathcal{L}_X$ is the Lie derivative in the direction of $X$ and $Ric^{\pm}(X)$ is the vector field on $M^n$ defined by
\begin{eqnarray*}
g^\pm(Ric^\pm(X),Y)=Ric^\pm(X,Y)
\end{eqnarray*} 
for all $Y\in\Gamma(TM^n)$. On the other hand, a straightforward (but lengthy) computation using the static equations (\ref{vacuum}) gives 
\begin{eqnarray}\label{RicMinus}
\frac{1}{2}{\rm Tf}\big(\mathcal{L}_Xg^\pm\big)=V^\pm Ric^\pm
\end{eqnarray}
where 
\begin{eqnarray*}
X=\frac{2^{\frac{4}{n-2}}}{(1-N^2)^{\frac{n}{n-2}}}\nabla N
\end{eqnarray*}
and ${\rm Tf}$ denotes the trace-free part of a symmetric tensor. Then putting together (\ref{Pohozaev-Pointwise}) and (\ref{RicMinus}), we deduce that
\begin{eqnarray*}
{\rm div}_{g^\pm}\Big(Ric^\pm(X)\Big)=V|Ric^\pm|^2_{g^\pm},
\end{eqnarray*}
and (\ref{DivIdenRem-1}) follows from the expression of the vector field $X$. It is important to note that the proof of these divergence identities relies on the full set of the static equations (\ref{vacuum}) and does not hold for pseudo-static systems.
\end{remark}


\section{Uniqueness without positive mass theorem}\label{UniquePMT}


In this section, we present a simple proof of Theorem \ref{BHU} which is clearly in the spirit of the original works \cite{I,MZH,R} but which puts forward a new geometric point of view. As we shall see, this method also applies when studying uniqueness questions for connected quasilocal photon surface in $3$-dimensional static vacuum triple. 

\subsection{The nondegenerate static horizon case}\label{UniquePMT-horizon}

If we assume that the assumptions of Theorem \ref{BHU} are fulfilled, one can integrate on $\partial M^3$ the inequality in Lemma \ref{TechnicalLemma} which, with the help of the Gauss-Bonnet formula, yields 
\begin{eqnarray*}
\pi\,\chi(\partial M^3)\geq 2\kappa^2 A.
\end{eqnarray*}
Here $\chi(\partial M^3)$ denotes the Euler characteristic of the nondegenerate static horizon. The formula (\ref{MassFormula}) for $n=3$ implies that the last inequality can be rewritten as 
\begin{eqnarray}\label{upper}
\frac{A}{32\pi}\chi(\partial M^3)\geq m^2.
\end{eqnarray}
On the other hand, a direct computation using the asymptotic of the lapse function $N$ gives that
\begin{eqnarray*}
W(p_\infty)=\frac{1}{16m^2}.
\end{eqnarray*}
Since the maximum of the smooth function $W$ defined on $M^3_\infty$ is reached on $\partial M^3$, we obtain  
\begin{eqnarray*}
\kappa^2\geq \frac{1}{16m^2},
\end{eqnarray*}
which from (\ref{MassFormula}), rewrites as the classical Penrose inequality
\begin{eqnarray*}
m^2\geq\frac{A}{16\pi}.
\end{eqnarray*}
Combining this inequality with (\ref{upper}) gives
\begin{eqnarray*}
\chi(\partial M^3)=2\quad\text{and}\quad m^2=\frac{A}{16\pi},
\end{eqnarray*}
and then $W$ has to be constant on $M^3_\infty$. In particular, $\partial M^3$ is homeomorphic to a $2$-sphere. From (\ref{DivIden}), $(M^3_\infty,g_\infty^-)$ is Ricci-flat and so flat since it is $3$-dimensional. This allows to conclude that $(M^3,g)$ is the desired exterior of the Schwarzschild manifold.

\subsection{The quasilocal photon surface case}\label{UniquePMT-photon}

We are now in position to give a similar proof of the following result:
\begin{theorem}\label{PSU}
An asymptotically flat static vacuum triple $(M^3,g,N)$ with a connected quasilocal photon surface as inner boundary is isometric to the exterior of a suitable piece of a Schwarzschild manifold of positive mass. 
\end{theorem}

First recall that since $\partial M^3$ is a quasilocal photon surface, it is a totally umbilical surface with positive constant mean curvature $H$ on which the lapse function is also a positive constant. Moreover, the identities (\ref{QuasiLocalPhotonSpheres1}) and (\ref{QuasiLocalPhotonSpheres2}) are fulfilled for a constant $c>1$. Then since $\nu(N)>0$, the formula (\ref{MassFormula}) implies that the mass of $(M^3,g)$ is non-zero. Moreover, from these properties, the Gauss formula reads as
\begin{eqnarray*}
2\,Ric(\nu,\nu)=\frac{1-c}{2}H^2
\end{eqnarray*}
so that computations as in Lemma \ref{TechnicalLemma} yields $N_0^2\leq c^{-1}$, where we let $N_0:=N_{|\partial M^3}$. Once again using (\ref{QuasiLocalPhotonSpheres1}), (\ref{QuasiLocalPhotonSpheres2}) and (\ref{MassFormula}), it is straightforward to check that this inequality is equivalent to 
\begin{eqnarray*}
m^2\leq \frac{1}{8}\Big(1-\frac{1}{c}\Big)^2\frac{A^2}{\omega_2^2}\Ss
\end{eqnarray*}
which, once integrated over $\partial M^3$, gives
\begin{eqnarray}\label{UpperMassPhoton}
m^2\leq \frac{1}{8}\Big(1-\frac{1}{c}\Big)^2\frac{A}{\omega_2}\chi(\partial M^3).
\end{eqnarray}
On the other hand, since the function $W$ reaches its maximum on the boundary (unless it is constant) we have
\begin{eqnarray*} 
W(p_\infty)=\frac{1}{16m^2}\leq W_{|\partial M^3}=\frac{\nu(N)^2}{(1-N_0^2)^4}\leq \Big(1-\frac{1}{c}\Big)^{-4}\frac{\omega_2^2}{A^2}m^2
\end{eqnarray*}
that is
\begin{eqnarray*}
m^2\geq\frac{1}{4}\Big(1-\frac{1}{c}\Big)^2\frac{A}{\omega_2}.
\end{eqnarray*}
Combining this inequality with (\ref{UpperMassPhoton}) allows to conclude that $\partial M^3$ is a topological $2$-sphere and that (\ref{UpperMassPhoton}) is in fact an equality. The maximum principle implies that $W$ is constant on the whole of $M^3$ and then the divergence identity (\ref{DivIden}) that $(M_\infty^3,g_\infty^-)$ is flat. Finally, it is not difficult to show that $(M_\infty^3,g_\infty^-)$ is isometric to a flat ball with radius $2/H^-$ and we can conclude that $(M^3,g)$ is isometric to the exterior $\{s\geq s_0\}$ (with $s_0=\widetilde{s}_l$ given in Section \ref{ProofMainTheorem}) of the Schwarzschild metric with mass 
\begin{eqnarray*}
m=\frac{1}{2}\Big(1-\frac{1}{c}\Big)\Big(\frac{A}{\omega_2}\Big)^{\frac{1}{2}}.
\end{eqnarray*}



\end{document}